%% file: Ueno.tex
\documentclass[12pt]{amsart}
\input{macros}

\usepackage{hyperref}
\usepackage{amsmath}
\usepackage{graphicx}

\title{A remark on the Ueno-Campana's threefold}

\author{Cinzia Bisi}
\address{Department of Mathematics and Computer Science\\ 
 University of Ferrara\\ Via Machiavelli, 35 \\ Ferrara 44121, Italy}\email{bsicnz@unife.it}

\author{Paolo Cascini}
\address{Department of Mathematics\\
Imperial College London\\
180 Queen's Gate\\
London SW7 2AZ, UK}
\email{p.cascini@imperial.ac.uk}

\author{Luca Tasin}
\address{Mathematical Institute of the University of Bonn \\ Endenicher Allee 60 D-53115 \\
Bonn, Germany.} 
\email{tasin@math.uni-bonn.de}

\thanks{Part of this work was completed while the second and third authors were attending the workshop "Algebraic Geometry" at Oberwolfach, on 16-20 March 2015. We would like to thank the organisers and the Institute for the invitation and for providing an ideal environment to work. We would also like to thank F. Catanese, J. Lesieutre, K. Oguiso and D.-Q. Zhang for many useful discussions. We are grateful to  the referee for carefully reading our manuscript and for suggesting several improvements.}
\thanks{The first author was partially supported by Prin 2010-2011 Protocollo: 2010NNBZ78-012, by Firb 2012 Codice: RBFR12W1AQ-001 and by GNSAGA-INdAM. The second author was funded by EPSRC. The third author was partially funded by the Italian grant GNSAGA-INdAM}
\dedicatory{Dedicated to Fabrizio Catanese on his 65th birthday}

\begin{document}

\begin{abstract}We show that the Ueno-Campana's threefold cannot be obtained as the blow-up of any smooth threefold along a smooth centre, answering negatively a question raised by 
Oguiso and  Truong.
\end{abstract}
\maketitle

\section{Introduction}

Let $E_{\tau}=\mathbb{C} / (\mathbb{Z} + \mathbb{Z} \tau)$ be the complex elliptic curve of period $\tau.$ 
There exist exactly two elliptic curves with automorphism group bigger than $\{\pm 1\}$: these are defined respectively by 
  the periods ${\sqrt{-1}}$ and the cubic root of  unity $\omega:= (-1+\sqrt{-3})/2$. 

We consider the diagonal action of the cyclic group generated by $\sqrt {-1}$ (resp. $-\omega$) on the product $$E_{\sqrt{-1}} \times E_{\sqrt{-1}} \times E_{\sqrt{-1}}\qquad  \text{(resp. } E_{\omega} \times E_{\omega} \times E_{\omega})$$ and we denote by $X_4$ (resp. $X_6$) the minimal resolution of their quotients: 
$$\begin{aligned}
E_{\sqrt{-1}} \times E_{\sqrt{-1}} \times E_{\sqrt{-1}} / \langle \sqrt{-1}  \rangle\qquad  
\text{(resp. } 
E_{\omega} \times E_{\omega} \times E_{\omega} / \langle -\omega   \rangle).
\end{aligned}$$
The minimal resolutions are obtained by a single blow-up at the maximal ideal of each singular point of the quotients above. 

The threefolds $X_4$ and $X_6$ have been extensively studied in the past. In particular, they admit an automorphism of positive entropy (e.g. see \cite{og14} for more details).
The variety $X_4$ is now referred as the {\em Ueno-Campana's threefold}. It has  been recently shown that $X_4$ and $X_6$ are rational. Indeed, Oguiso and Truong  \cite{ot15} showed the rationality of  $X_6$, and  Colliot-Th\'el\'ene \cite{Colliot15} showed the rationality of  $X_4$, after the work of Catanese, Oguiso and Truong \cite{cot14}. The unirationality of $X_4$ was conjectured by Ueno \cite{Ueno75}, whilst 
 Campana asked about the rationality of $X_4$ in  \cite{Ca11}.

\medskip

The aim of this note is to give a negative answer to the following question raised by  Oguiso and Truong (see \cite{og14}[Question 5.11] and \cite{ttt15}[Question 2]).

\begin{question}\label{q_oguiso}
Can $X_4$ or $X_6$ be obtained as the blow-up of $\mathbb P^3$, $\mathbb P^2\times \mathbb P^1$ or $\mathbb P^1\times \mathbb P^1\times \mathbb P^1$ along smooth centres?
\end{question}

Our main result is the following:

\begin{theorem}\label{main}
Let $A$ be an abelian variety of dimension three and let $G$ be a finite group acting on $A$ such that the quotient map 
$$\rho\colon A \to Z=A/G$$ is \'etale in codimension 2. 

Assume that there exists a  resolution $f\colon X \to Z$  given by the blow-up of the singular points of $Z$ and such that the exceptional divisor at each singular point of $Z$ is irreducible. 

Then $X$ cannot be obtained as the blow-up of a smooth threefold along a smooth centre.  
\end{theorem}

Note that Theorem \ref{main} provides a negative answer to Question \ref{q_oguiso}. Very recently, Lesieutre \cite{lesieutre15} announced that Question \ref{q_oguiso} admits a negative answer, using different methods.

\section{Preliminary results}

We use some of the methods introduced in \cite{ct14}. Let $X$ be a normal projective threefold with isolated quotient singularities. Given a basis $\gamma_1,\dots,\gamma_m$ of $H^2(X,\C)$, the {\em cubic form associated to} $X$ is the homogeneous polynomial of degree $3$ defined by: 
$$F_X(x_1,\dots,x_m)=(x_1\gamma_1+\dots+x_m\gamma_m)^3\in \C[x_1,\dots,x_m].$$
Note that, modulo the natural action of $\GL(m,\C)$, the cubic $F_X$ does not depend on the choice of the base and it is a topological invariant of the underlying manifold $X$ (see \cite{ov95} for more details). 
In particular, if 
$$\mathcal H_{F_X}=(\partial_{x_i}\partial_{x_j}F_X)_{i,j=1,\dots,m}$$ 
denotes the {\em Hessian matrix} associated to $F_X$ and $p\in H^2(X,\C)$, then the rank of ${\mathcal H}_{F_X}$ at $p$ is well-defined. 

The following basic tool was used in \cite{ct14} in a more general context. We provide a proof for the reader's convenience. 

\begin{lemma}\label{rk2} Let $Y$ be a normal projective threefold with isolated quotient singularities and let $f\colon X\to Y$ be the blow-up of $Y$ along  a point $q\in Y$ (resp. a curve $C\subseteq Y$). Assume that   the exceptional divisor of $f$  is irreducible and let $E$ be its class in $H^2(X,\mathbb C)$. 

Then the rank of the Hessian matrix $\mathcal H_{F_X}$ of $F_X$ at $E$ is one (resp. at most two). 
\end{lemma}

Note that by \cite{ct14}[Lemma 2.7 and Lemma 2.12] the rank of $\mathcal H_{F_X}$ is never zero. 
\begin{proof}
We have $H^2(X, \C) =\langle E, f^*(\gamma_1), \ldots , f^*(\gamma_m) \rangle$ where $\gamma_1, \ldots, \gamma_m$ is a basis  of $ H^2 (Y, \mathbb{C})$.

Consider the cubic form $F_X$ associated to $X$ with respect to this basis:
$$
F_X (x_0, \ldots ,x_m)=(x_0 E + \sum\limits_{i=1}^m x_i f^* (\gamma_i))^3.  
$$
Since $f^*(\gamma_i) \cdot f^*(\gamma_j)\cdot E= 0$  for all $i,j=1, \ldots m,$  we have
$$\begin{aligned}
F_X (x_0, \ldots , x_m)
=& x_0^3 E^3 + 3\sum\limits_{i=1}^m x_0^2 x_i  E^2 f ^* (\gamma_i)  + (\sum\limits_{i=1}^m x_i f^* (\gamma_i))^3.  
\end{aligned}
$$
Let $a=E^3$ and let $b_i=E^2 f ^* (\gamma_i)$ for $i=1,\dots,m$. Note that  if $f$ is the blow-up of a point $q\in Y$ then $b_1=\ldots=b_m=0$. 

Thus, we have
$$
F_X(x_0, \ldots, x_m)=ax_0^3 + 3\sum\limits_{i=1}^m b_i x_0^2 x_i + G(x_1, \ldots , x_m),
$$
where   $G$ is a homogeneous cubic polynomial in the  variables $x_1, \ldots , x_m,$ i.e. it does not depend on $x_0.$ 
Let $p=y_0E+\sum_{i=1}^m y_i f^*\gamma_i\in H^2(X,\mathbb C)$, for some $y_0,\dots,y_m\in \mathbb C$ and let $p'=(y_1,\dots,y_m)$. After removing the first row and the first column, the Hessian matrix  $\mathcal{H}_{F_X}(p)$ of $F_X$ at $p$, coincides with the Hessian matrix  $\mathcal{H}_G(p')$ of $G$ at $p'$.  

In particular, if $p=E$, then $p'=(0,\dots,0)$ and $\mathcal H_G(p')$ is the zero matrix. Thus,  the rank of  the Hessian of $F_X$ at $p$ is at most two. In addition, if $b_1=\ldots=b_m=0$, then the rank of $\mathcal H_F$ at $p$ is exactly one. 
 \end{proof}

\section{Proofs}

\begin{lemma}\label{abelian}
Let $A$ be an abelian variety of dimension $3$ and let $G$ be a finite group acting on $A$ such that the quotient map $\rho\colon A \to Z=A/G$ is \'etale in codimension 2. Let $F_Z$ be the cubic form associated to $Z$ and let $p \in H^2(Z,\C)$ such that $\rk \m H_{F_Z}(p) \le 1$. 

Then $p=0$.
\end{lemma}
\begin{proof}
The morphism $\rho$ induces an immersion of vector spaces 
$$\rho^* \colon H^2(Z, \C) \to H^2(A,\C).$$  Thus, there exists a  basis of $H^2(A,\C)$ such that if $F_A$ is the cubic associated to $A$ with respect to this basis and $d$ is the degree of $\rho$, then 
$$
F_Z(x_1,\dots,x_m)=d \cdot F_A(x_1,\dots,x_m,0,\dots,0).
$$

It is enough to show that if $q\in H^2(A,\C)$ is such that the rank of $\mathcal H_{F_A}$ at $q$ is not greater than one, then $q=0$. 

Write $A= \C^3 / \Gamma$ and consider $z_1, z_2, z_3$ coordinates on $\C^3$. Then a basis of $H^2(A,\C)$ is given by 

\begin{align*}
z_{ij} &= dz_i \wedge dz_j \quad 1 \le i <j \le 3, \\
z_{i \bar j} &= dz_i \wedge d \bar{z_j} \quad  i \,,  j \in \{ 1, 2, 3 \}, \\
z_{\bar i \bar j} &= d\bar z_i \wedge d\bar z_j \quad 1 \le i <j \le 3.
\end{align*}

For any $x\in H^2(A,\C)$, let $x_{ij}, x_{i\bar j}$ and $x_{\bar i \bar j}$ be the coordinates of $x$ with respect to the basis above and let $F'_A$ be the cubic associated to this basis. It is enough to show that if $q\in H^2(A,\C)$ is such that the rank of $\m H_{F'_A}$ at $q$ is not greater than one, then $q=0$. Let $q_{ij}, q_{i\bar j}$ and $q_{\bar i \bar j}$ be the coordinates of $q$. 

The  $(2 \times 2)$-minor of $\m H_{F'_A}$ at $x$ defined by the rows corresponding to $x_{12}$ and $x_{13}$ and the columns corresponding to $x_{2\bar 1}$ and $x_{3\bar 1}$ is given by 
$$
\begin{pmatrix}
0 & 6x_{\bar 2 \bar 3} \\
6x_{\bar 2 \bar 3} & 0
\end{pmatrix}.
$$
It follows that $q_{\bar 2 \bar 3}=0$. 
%
%
By choosing suitable $(2\times 2)$-minors, it follows easily that each coordinate of $q$ is zero.  Thus, the claim follows. 
\end{proof}

\begin{proof}[Proof of Theorem \ref{main}]
Suppose not. Then there exists a smooth projective threefold $Y$ such that $X$ can be obtained as the blow-up $g\colon X\to Y$ at a smooth centre.  Let $E$ be the exceptional divisor of $g$. Let $k$ be the number of singular points of $Z$ and let  $E_1,\dots,E_k$ be the exceptional divisors on $X$ corresponding to the singular points of $Z$. 

We want to prove that $E=E_i$ for some $i=1,\dots,k$. Denote by $p$ the class of $E$ in $H^2(X,\C)$. Lemma \ref{rk2} implies that the rank of $\m H_{F_X}$ at $p$ is not greater than two.

Let $\gamma_1,\dots,\gamma_{m}\in H^2(Z,\C)$ be a basis  and let $F_Z$ be the associated cubic form. Then $f^*\gamma_1,\dots,f^*\gamma_m,[E_1],\dots,[E_k]$ is a basis of $H^2(X,\C)$ and if $F_X$ denotes the associated cubic form, we have 
$$
F_X(x_1, \ldots, x_{m}, y_1, \ldots, y_k)= F_Z(x_1, \ldots, x_{m}) + \sum_{i=1}^{k} a_iy_i^3,
$$
where $a_i=E^3_i$ is a non-zero integer, for $i=1,\dots,k$. 

Thus, the Hessian matrix of $F_X$ is  composed by two blocks: one is the Hessian matrix of $F_Z$ and the other one is a diagonal matrix, whose only non-zero entries are $6a_i$ for $i=1, \ldots, k$. We may write  $p=(p^0, p^1)=(p^0_1, \ldots, p^0_{m}, p^1_1, \ldots, p^1_k)$. We  have 
$\rk \m H_{F_Z}(p^0) \le 2$.

We distinguish two cases. If $\rk \m H_{F_Z}(p^0) =2$, then $p^1=(0, \ldots, 0)$ and in particular $E$ is numerically equivalent to  $f^* D$, for some  pseudo-effective Cartier divisor $D$ on $Z$. 
Since $A$ is abelian, it follows that $\rho^* D$ is a nef divisor. Thus $E$ is nef, a contradiction.

If $\rk \m H_{F_Z}(p^0) \le 1$, then  Lemma \ref{abelian} implies that $p^0=0$. Thus,
$$
E \equiv c_s E_s + c_t E_t
$$
for some distinct $s, t\in \{1,\dots,k\}$ and $c_s, c_t$  rational numbers. Since $E$ is effective non-trivial, at least one of the $c_i$ is positive. By symmetry, we may assume $c_s >0$. By the negativity lemma, the divisor $E_s$ is covered by rational curves $C$ such that $E_s\cdot C <0$. Since $E_s$ and $E_t$ are disjoint, it follows that $E\cdot C <0$, which implies that $C$ is contained in $E$. Thus $E_s$ is contained in  $E$. Since $E$ is prime, it follows that $E=E_s$ and $c_t=0$. 

Finally, note that $g$ contracts $E=E_s$ to a point, as otherwise there exists a small contraction $\eta\colon Y\to Z$ and in particular $Z$ is not $\mathbb Q$-factorial, a contradiction. 
Thus, $g\colon X\to Y$ is the contraction of $E_s$ to the corresponding singular point on $Z$, which is again a contradiction. 
The claim follows.\end{proof}
 
\begin{remark} As K. Oguiso kindly pointed out to us, the same proof shows that if $f\colon X\to Z$ is as in Theorem \ref{main} and $g$ is an automorphism on $X$ then the set of exceptional divisors of $f$ is invariant with respect to $g$. Thus, there exists a positive integer $m$ such that the power $g^m$  descends to an automorphism on $Z$.
 \end{remark}
\bibliographystyle{amsalpha}
\bibliography{Library}
\end{document}

%% file: macros.tex
\usepackage{amsfonts,amsmath,amsthm,amssymb,latexsym,amsthm,newlfont,enumerate,color}

\newif\ifslide

\theoremstyle{plain}

\ifdefined\spacer
\newtheorem{theorem}{Theorem}
 \else
\newtheorem{theorem}{Theorem}[section]
\fi

\newtheorem{lemma}[theorem]{Lemma}

\newtheorem{definition-lemma}[theorem]{Definition-Lemma}
\newtheorem{question}[theorem]{Question}
\newtheorem{red-question}[theorem]{\textcolor{red}{Question}}

\theoremstyle{definition}

\newtheorem{remark}[theorem]{Remark}

\def\ideal#1.{I_{#1}}
\def\ring#1.{\mathcal {O}_{#1}}

\def\C{\mathbb C}
\def\m#1{\mathcal{#1}}

\def\fring#1.{\hat{\mathcal {O}}_{#1}}
\def\proj#1.{\mathbb {P}(#1)}
\def\pr #1.{\mathbb {P}^{#1}}
\def\dpr #1.{\hat{\mathbb {P}}^{#1}}
\def\af #1.{\mathbb A^{#1}}
\def\Hz #1.{\mathbb F_{#1}}
\def\Hbz #1.{\overline{\mathbb F}_{#1}}
\def\fb#1.{\underset #1 {\times}}
\def\rest#1.{\underset {\ \ring #1.} \to \otimes}
\def\au#1.{\operatorname {Aut}\,(#1)}
\def\deg#1.{\operatorname {deg } (#1)}

\def\pic#1.{\operatorname {Pic}\,(#1)}
\def\pico#1.{\operatorname{Pic}^0(#1)}
\def\picg#1.{\operatorname {Pic}^G(#1)}
\def\ner#1.{NS (#1)}
\def\rdown#1.{\llcorner#1\lrcorner}
\def\rfdown#1.{\lfloor{#1}\rfloor}
\def\rup#1.{\ulcorner{#1}\urcorner}
\def\rcup#1.{\lceil{#1}\rceil}

\def\n1#1.{\operatorname {N_1}(#1)}  
\def\cn1#1.{\overline{\operatorname {N^1}(#1)}} 
\def\cone#1.{\operatorname {NE}(#1)}     
\def\ccone#1.{\overline{\operatorname {NE}}(#1)}
\def\none#1.{\operatorname {NF}(#1)}
\def\cnone#1.{\overline{\operatorname {NF}}(#1)}
\def\mone#1.{\operatorname {NM}(#1)} 
\def\cmone#1.{\overline{\operatorname {NM}}(#1)}

\def\coef#1.{\frac{(#1-1)}{#1}}
\def\vit#1.{D_{\langle #1 \rangle}}
\def\mm#1.{\overline {M}_{0,#1}}
\def\H1#1.{H^1(#1,{\ring #1.})}
\def\ac#1.{\overline {\mathbb F}_{#1}}

\def\adj#1.{\frac {#1-1}{#1}}
\def\spn#1.{\overline{#1}}
\def\pek#1.#2.{\Cal P^{#1}(#2)}
\def\plk#1.#2.{\Cal P^{\leq #1}(#2)}
\def\ev#1.{\operatorname{ev_{#1}}}
\def\ilist#1.{{#1}_1,{#1}_2,\dots}
\def\bminv#1.{(\nu_1,s_1;\nu_2,s_2;\dots ;\nu_{#1},s_{#1};\nu_{r+1})}
\def\zinv#1.{(\nu_1,s_1;\nu_2,s_2;\dots ;\nu_{#1},s_{#1};0)}
\def\iinv#1.{(\nu_1,s_1;\nu_2,s_2;\dots ;\nu_{#1},s_{#1};\infty)}

\def\scr #1.{\mathcal #1}

\def\llist#1.#2.{{#1}_1,{#1}_2,\dots,{#1}_{#2}}
\def\ulist#1.#2.{{#1}^1,{#1}^2,\dots,{#1}^{#2}}
\def\lomitlist#1.#2.{{#1}_1,{#1}_2,\dots,\hat {{#1}_i}, \dots, {#1}_{#2}}
\def\lomitlistz#1.#2.{{#1}_0,{#1}_1,\dots,\hat {{#1}_i}, \dots, {#1}_{#2}}
\def\loc#1.#2.{\Cal O_{#1,#2}}
\def\fderiv#1.#2.{\frac {\partial #1}{\partial #2}}
\def\deriv#1.#2.{\frac {d #1}{d #2}}

\def\map#1.#2.{#1 \longrightarrow #2}
\def\rmap#1.#2.{#1 \dasharrow #2}
\def\emb#1.#2.{#1 \hookrightarrow #2}
\def\non#1.#2.{\text {Spec }#1[\epsilon]/(\epsilon)^{#2}}
\def\Hi#1.#2.{\text {Hilb}^{#1}(#2)}
\def\sym#1.#2.{\operatorname {Sym}^{#1}(#2)}
\def\Hb#1.#2.{\text {Hilb}_{#1}(#2)}
\def\Hm#1.#2.{\Hom_{#1}(#2)}
\def\prd#1.#2.{{#1}_1\cdot {#1}_2\cdots {#1}_{#2}}
\def\Bl #1.#2.{\operatorname {Bl}_{#1}#2}
\def\pl #1.#2.{#1^{\otimes #2}}
\def\mgn#1.#2.{\overline {M}_{#1,#2}}
\def\ialist#1.#2.{{#1}_1 #2 {#1}_2, #2\dots}
\def\pair#1.#2.{\langle #1, #2\rangle}
\def\vandermonde#1.#2.{\left|
\begin{matrix}
1 & 1 & 1 & \dots & 1\\
{#1}_1 & {#1}_2 & {#1}_3 & \dots & {#1}_{#2}\\
{#1}_1^2 & {#1}_2^2 & {#1}_3^2 & \dots & {#1}_{#2}^2\\
\vdots & \vdots & \vdots & \ddots & \vdots\\
{#1}_1^{#2-1} & {#1}_2^{#2-1} & {#1}_2^{#2-1} & \dots & {#1}_{#2}^{#2-1}\\
\end{matrix}
\right|
}
\def\vandermondet#1.#2.{\left|
\begin{matrix}
1 & {#1}_1   & {#1}_1^2 & \dots & {#1}_1^{#2-1}\\
1 & {#1}_2   & {#1}_2^2 & \dots & {#1}_2^{#2-1}\\
1 & {#1}_3   & {#1}_3^2 & \dots & {#1}_3^{#2-1}\\
\vdots & \vdots & \vdots & \ddots & \vdots\\
1 & {#1}_{#2}& {#1}_{#2}^2 & \dots & {#1}_{#2}^{#2-1}\\
\end{matrix}
\right|
}
\def\gr#1.#2.{\mathbb{G}(#1,#2)}

\def\alist#1.#2.#3.{{#1}_1 #2 {#1}_2 #2\dots #2 {#1}_{#3}}
\def\zlist#1.#2.#3.{#1_0 #2 #1_1 #2\dots #2 #1_{#3}}
\def\lomitlist30#1.#2.#3.{{#1}_0,{#1}_1 #2 \dots #2\hat {{#1}_i} #2\dots #2 {#1}_{#3}}
\def\lmap#1.#2.#3.{#1 \overset{#2}{\longrightarrow} #3}
\def\mes#1.#2.#3.{#1 \longrightarrow #2 \longrightarrow #3}
\def\ses#1.#2.#3.{0\longrightarrow #1 \longrightarrow #2 \longrightarrow #3 \longrightarrow 0}
\def\les#1.#2.#3.{0\longrightarrow #1 \longrightarrow #2 \longrightarrow #3}
\def\res#1.#2.#3.{#1 \longrightarrow #2 \longrightarrow #3\longrightarrow 0}
\def\Hi#1.#2.#3.{\text {Hilb}^{#1}_{#2}(#3)}
\def\ten#1.#2.#3.{#1\underset {#2}{\otimes} #3}
\def\lomitlist30#1.#2.#3.{{#1}_0 #2 {#1}_1 #2 \dots #2 \hat {{#1}_i} #2 \dots #2 {#1}_{#3}}
\def\mderiv#1.#2.#3.{\frac {d^{#3} #1}{d #2^{#3}}}

\def\Hom{\operatorname{Hom}}

\def\deg{\operatorname{deg}}

\def\GL{\operatorname{GL}}

\def\rk{\operatorname{rk}}

\def\rest{\operatorname{res}}

\def\C{\mathbb C}

\def\e{\Cal E}

\def\e1{E_1}
\def\e2{E_2}

\def\mapdown#1{\big\downarrow\rlap{$\vcenter{\hbox{$\scriptstyle#1$}}$}}

\def\mapse#1{
{\vcenter{\hbox{$\mathop{\smash{\raise1pt\hbox{$\diagdown$}\!\lower7pt
\hbox{$\searrow$}}\vphantom{p}}\limits_{#1}\vphantom{\mapdown{}}$}}}}

\def\VR#1.{height#1pt&\omit&&\omit&&\omit&&\omit&&\omit&\cr}

\def\VRT#1.{height#1pt&\omit&&\omit&\cr}
